\documentclass[12pt]{amsart}
\usepackage{amsmath,amsthm,amsfonts,latexsym,amssymb,amscd,xcolor}
\usepackage{chemarr}
\usepackage{caption}
\usepackage{mathtools}
\usepackage{tikz}
\newtheorem{thm}{Theorem}[section] %[subsection]

\newtheorem{lm}[thm]{Lemma}

\newtheorem{clm}[thm]{Claim}
\theoremstyle{definition}

\numberwithin{equation}{section}
\newcommand{\cproof}{\noindent{\it Proof of Claim.}\ } %used in spf environment
\newcommand{\cqed}{\hfill\rule{1.3mm}{3mm}}

\newcommand{\wec}[1]{{\mathbf{#1}}}  % notation for vectors in algebras
\newcommand{\m}[1]{{\mathbf{\uppercase{#1}}}}

\DeclareMathOperator{\Con}{Con}

\newcommand{\cg}{\mathrm{Cg}}

\begin{document}

\title[Locally finite Schreier varieties]{Locally finite Schreier varieties}

\author{K.~A.~Kearnes}
\address[Keith A. Kearnes]{Department of Mathematics\\
University of Colorado\\
Boulder, CO 80309-0395\\
USA}
\email{kearnes@colorado.edu}

\author{A.~Moorhead}
\address[Andrew Moorhead]{Institut f{\" u}r Algebra,
TU Dresden}
\email{apmoorhead@gmail.com}  

\author{\'A.~Szendrei}
\address[\'Agnes Szendrei]{Department of Mathematics\\
University of Colorado\\
Boulder, CO 80309-0395\\
USA}
\email{szendrei@colorado.edu}

\thanks{The second author received funding from the ERC
  (Grant Agreement no. 101071674, POCOCOP).
  Views and opinions expressed are however
  those of the authors only and do not necessarily
  reflect those of the European Union or the European Research
Council Executive Agency.}

\subjclass[2010]{Primary: 08B20; Secondary: 08A30}
\keywords{Free algebra, Nielsen variety,
  P\'{a}lfy's Theorem, permutational algebra, Schreier variety,
  Tame Congruence Theory, $\langle 0,1\rangle$-minimal algebra}

\begin{abstract}
A variety $\mathcal{V}$ is called a \emph{Schreier variety}
if every subalgebra of a $\mathcal{V}$-free
algebra is a $\mathcal{V}$-free algebra.
We use ideas from Tame Congruence Theory
to classify locally finite Schreier varieties.
One version of the classification theorem states that a locally
finite variety $\mathcal{V}$
is a Schreier variety if and only
if (i) every finite algebra in $\mathcal{V}$
is a $\langle 0,1\rangle$-minimal algebra
and (ii) if $\mathcal{V}$ has a constant
$1$-ary term operation, then $\mathcal{V}$ also has
a constant $0$-ary term operation.
\end{abstract}

\maketitle

\section{Introduction}\label{Ind_sec}
In 1927, in \cite{schreier}, Otto Schreier
proved that every
subgroup of a free group is free.
This theorem extended a 1921 result of
Jakob Nielsen who proved in
\cite{nielsen}
that every \emph{finitely generated}
subgroup of a free group is free.
Let us say that a variety $\mathcal V$
of algebraic structures
is a \textbf{Schreier variety}
if every subalgebra of a $\mathcal V$-free
algebra is a $\mathcal V$-free algebra, and 
let us say that a variety $\mathcal V$
is a \textbf{Nielsen variety}
if every \emph{finitely generated}
subalgebra of a $\mathcal V$-free algebra is a $\mathcal V$-free algebra.
Since a trivial variety (i.e., a variety satisfying $x\approx y$)
has the property
that all members are free, we may focus our
attention on nontrivial varieties when we discuss
these concepts.

In this paper, we will prove that a locally finite
variety is a 
Schreier variety if and only if it is a Nielsen variety
and we will classify these varieties (Theorem~\ref{main_thm}).
In this theorem, a zeroary (= $0$-ary) operation of a variety
may be called a ``constant'' operation, while a
unary (= $1$-ary) operation that interprets as a constant
unary function throughout the variety
will be called a ``pseudoconstant''.
We shall explain why a variety with at least one pseudoconstant operation
but no constant operation cannot be a Schreier or Nielsen variety,
so such varieties are excluded by hypothesis from our main theorem.
Given this restriction, the main theorem of the paper asserts that
the following four statements are equivalent for any
other locally finite
variety $\mathcal{V}$:
\begin{enumerate}
\item $\mathcal{V}$ is a Schreier variety.
\item $\mathcal{V}$ is a Nielsen variety.
\item Every finite algebra in $\mathcal{V}$ is a
  $\langle 0,1\rangle$-minimal algebra.
\item Either there is a finite group $G$ such that
  every algebra in $\mathcal{V}$ is polynomially
  equivalent to a $G$-set, or
there is a finite field $\mathbb{F}$ such that
  every algebra in $\mathcal{V}$ is polynomially
  equivalent to an $\mathbb{F}$-vector space.
\end{enumerate}

\section{Constants, pseudoconstants, and
  $\m f_{\mathcal V}(\emptyset)$} \label{FV(0)}

The questions we consider (\emph{Which locally finite varieties
are Schreier or Nielsen?}) belong to the subject of
categorical algebra.
From the viewpoint of category theory,
a variety $\mathcal V$ arises from a monad on the category of sets,
hence
has the property
that for any set $X$ there is a free algebra over $X$.
This free algebra over $X$ in $\mathcal{V}$
will be denoted $\m f_{\mathcal{V}}(X)$.
We will write $\m f_{\mathcal{V}}(\kappa)$ to denote any
algebra that is free over some set $X$ satisfying $|X|=\kappa$.
The cardinal $\kappa$ is the called the {\bf rank} of
$\m f_{\mathcal{V}}(\kappa)$.

The subset $X\subseteq F_{\mathcal{V}}(X)$ of free
generators is characterized by the property
that it is 
a $\mathcal{V}$-independent generating subset for
$\m f_{\mathcal{V}}(X)$. If $Y\subseteq X$, then
the subalgebra $\m s\leq \m f_{\mathcal{V}}(X)$ that
is generated by $Y$ will be a
$\mathcal{V}$-independent generating subset for
$\m s$, so $\m s$ will be free of rank $|Y|$ in $\mathcal{V}$.
In the case where $Y=\emptyset$,
the free algebra $\m f_{\mathcal V}(\emptyset)$ must be the initial
object of $\mathcal V$. This algebra
will be empty if the signature of $\mathcal{V}$
contains no $0$-ary operation symbols.
But, it is a common convention in algebra
that every algebra has a nonempty underlying set.
(E.g., see 
\cite[Definition~1.1]{bergman},
\cite[Definition~1.3]{burris}, or
\cite[Definition~1.1]{mmt}.)
This leads to a conflict:
if you exclude empty algebras from varieties,
then you might also exclude initial objects,
in which case you exclude $\m f_{\mathcal V}(\emptyset)$.

Because of the observations of the previous paragraph,
we shall allow algebras to be empty
and we shall allow varieties to contain empty algebras.
Of course, if the signature of the variety contains constant
symbols, then none of the algebras can be empty.
But if the signature of a variety $\mathcal V$ contains no constant
symbols, then there will be an empty algebra in $\mathcal V$ and
its isomorphism type will be unique.
When a variety contains an empty algebra,
then the empty algebra is the only candidate for
the initial object, hence the empty algebra will be
$\m f_{\mathcal V}(\emptyset)$.

The choice we make to admit empty algebras
leads to an odd-sounding conclusion.
If a ``group'' is presented as an
algebra of the form $\langle G; \cdot, {}^{-1},1\rangle$
in the signature of one binary operation,
one unary operation, and one constant, then
Schreier's Theorem is true: subgroups of free groups
are free. But if a ``group'' is instead presented as an 
algebra of the form $\langle G; \cdot, {}^{-1}\rangle$
in the signature of one binary operation and
one unary operation only, then
Schreier's Theorem is false: the $1$-element
group is no longer free, but it arises as a subgroup
of a free group of any nonzero rank.

Let's elaborate on the example in the previous paragraph.
Let $L$ be a first-order algebraic language, 
by which we mean a first-order language with
operation symbols and without relation symbols.
An \textbf{algebra} is an $L$-structure.
A \textbf{variety} is an equationally defined
class of algebras.
A $0$-ary operation symbol of $L$ will be called
\textbf{constant symbol}.
An $L$-term with no free variables will be called 
a \textbf{constant $L$-term}.
If $c$ is a constant $L$-term and $\m a$
is an $L$-algebra, then we may say that
the interpretation $c^{\m a}\in A$ is a \textbf{true constant}
of $\m a$.
A unary $L$-term $u(x)$ is a \textbf{pseudoconstant}
for an algebra $\m a$ if $\m a\models u(x)\approx u(y)$
and is a \textbf{pseudoconstant}
for a variety $\mathcal V$ if ${\mathcal V}\models u(x)\approx u(y)$.
The following assertions are consequences
of the way terms are defined and interpreted:
Let $\mathcal V$ be a variety in the language $L$.
\begin{enumerate}
\item[(1)] If $c$ is a constant $L$-term, then
  there is a pseudoconstant $u(x)$ whose
  interpretation in any $\m a\in {\mathcal V}$
  is a constant unary function with range $\{c^{\m a}\}$.
  \smallskip

\item[(2)] If $u(x)$ is 
  a pseudoconstant for $\mathcal V$
  and $c$ is any constant $L$-term,
  then $u(c)$ is also a constant $L$-term
and for any $\m a\in {\mathcal V}$   we have that
$u^{\m a}(x)$ is a constant unary function
with range $\{u^{\m a}(c^{\m a})\}$.
  \smallskip

\item[(3)] For any $\m a\in {\mathcal V}$ the set of
true constants of $\m a$
is a subuniverse $\m c_{\m a}$ of $\m a$.
The set of elements of $\m a$ belonging to the range
of a pseudoconstant is
a subuniverse $\m p_{\m a}$ of $\m a$.
  \smallskip

\item[(4)]
$\m c_{\m a}\subseteq \m p_{\m a}$ for all $\m a\in {\mathcal V}$.
  If $L$ contains at least one constant symbol, then
$\m c_{\m a}=\m p_{\m a}$  for all $\m a\in {\mathcal V}$.
  If $L$ does not contain a constant symbol, then
  $\m c_{\m a}=\emptyset$ while
  $\m p_{\m a}$  may or may not be empty.
  \smallskip

\item[(5)]
  If $X\neq \emptyset$ and
  the free algebra $\m f=\m f_{\mathcal V}(X)$ is nontrivial,
  then for any $x\in X$ we have
  $x\notin \m p_{\m f}$.
  (If $x=u(x)$ for some pseudoconstant $u(x)$ of $\m f$, then
  $\m f\models x\approx u(x)\approx u(y)\approx y$ contradicting
  the nontriviality of $\m f$.)
\end{enumerate}
  \smallskip

If $\mathcal V$ is a Schreier or Nielsen variety, then
for any nontrivial free algebra $\m f = \m f_{\mathcal V}(X)$ with
$X\neq\emptyset$,
the subalgebra $\m p_{\m f}$ will be a finitely generated
subalgebra of a free algebra, hence $\m p_{\m f}$ will be free
in $\mathcal{V}$.
By Item (5) above, no free algebra of positive rank
can consist entirely of elements in
the ranges of pseudoconstants,
so $\m p_{\m f}$ must be free of rank zero.
If the language contains no constant symbols,
then $\m f_{\mathcal V}(\emptyset)$ will be empty.
This shows that, in a Schreier or Nielsen variety,
if $\m c_{\m f}=\emptyset$, then $\m p_{\m f}=\emptyset$.
Applying Item~(4) above, this means that if $\mathcal{V}$
is a Schreier or Nielsen variety, then $\m c_{\m f}=\m p_{\m f}$
for every free algebra $\m f\in \mathcal{V}$.
Equivalently, 

\begin{quote}{\it
  If $\mathcal V$ is a Schreier or Nielsen variety in a language with
no constant symbols, then $\mathcal V$ has no pseudoconstants.}
\end{quote}

Using Item (2) above, we may restate this in a stronger way as

\begin{quote}
{\it If $\mathcal V$ is a Schreier or Nielsen variety, then
for any pseudoconstant $u(x)$ of $\mathcal V$
there is a constant term $c$ such that for every $\m a\in \mathcal V$
the function 
$u^{\m a}(x)$ is constant with range $\{c^{\m a}\}$.}
\end{quote}

The discussion of this section concludes
now with the declaration that the questions
we consider (\emph{Which locally finite varieties
are Schreier or Nielsen?}) are uninteresting for
varieties that have pseudoconstants
but no constants, because
(under our conventions) no Schreier or Nielsen
variety has a pseudoconstant 
unless it also has a constant.
Varieties with pseudoconstants 
but no constants will be excluded
by hypothesis in our main theorem.
(For example, the variety of groups
in the signature $\{\cdot, {}^{-1}\}$
of one binary operation and one unary operation
would be excluded by this hypothesis, since $x\cdot x^{-1}$
is a pseudoconstant for this variety, but this
language has no constants.)

\section{Free algebras in
  $\langle 0,1\rangle$-Minimal Varieties} \label{decomp}

In \cite{palfy}, P\'{a}lfy calls
an algebra $\m a$ \textbf{permutational}
if every unary polynomial operation of $\m a$ is a permutation
of the universe $A$ or is a constant operation.
Finite permutational algebras
are the fundamental objects in the structure
theory of finite algebras
that is known as Tame Congruence Theory, \cite{HM},
where they are also called \textbf{$\langle 0,1\rangle$-minimal algebras}
or just \textbf{minimal algebras}
(see Definition~2.14 of \cite{HM}).

It follows from the main theorem of
\cite{palfy} (or from the material of Chapter~4 of \cite{HM}) that
a (finite) $\langle 0,1\rangle$-minimal
algebra is one of the following types of algebras:
\smallskip

\begin{enumerate}
\item[\textbf{1}.] an algebra polynomially equivalent to a faithful $G$-set
for some finite group $G$,
\smallskip

\item[\textbf{2}.] an algebra polynomially equivalent to an
  $\mathbb F$-vector space for some finite field $\mathbb F$, 
  \smallskip

\item[\textbf{3}.] an algebra polynomially equivalent
  to a $2$-element Boolean algebra,
  \smallskip

\item[\textbf{4}.] an algebra polynomially equivalent
  to a $2$-element lattice, or
  \smallskip

\item[\textbf{5}.] an algebra polynomially equivalent
  to a $2$-element semilattice.
\end{enumerate}
\smallskip

\noindent
Here, we use the numbering system $\mathbf{1}$-$\mathbf{5}$ for
`types' that comes from Tame Congruence Theory.
In this paper, the last three types will
be handled by the same arguments, so we do not
need to distinguish between them. Instead,
we will refer to types $\mathbf{1}$ and $\mathbf{2}$
(the abelian types) and types $\mathbf{3}$-$\mathbf{5}$
(the nonabelian types).

We will be examining locally finite varieties $\mathcal V$
whose finite members are $\langle 0,1\rangle$-minimal.
This means that ${\mathcal V}_{\textrm{fin}}$,
the class of finite algebras in $\mathcal{V}$,
consists exclusively of algebras of types $\mathbf{1}$, $\mathbf{2}$,
or $\mathbf{3}$-$\mathbf{5}$.
In this paper, we will call a
variety $\mathcal{V}$ a \textbf{permutational variety}, or a
\textbf{$\langle 0,1\rangle$-minimal variety},
if all of its finite members are
permutational (=$\langle 0,1\rangle$-minimal).
For such varieties there is a uniformizing principle
with regard to the types of algebras
that can appear in ${\mathcal V}_{\textrm{fin}}$.
The principle is based
on the observation that if $\m a, \m b\in {\mathcal V}_{\textrm{fin}}$
are nontrivial, then 
since $\m a\times \m b\in {\mathcal V}_{\textrm{fin}}$
must be of type $\mathbf{1}$, $\mathbf{2}$, or $\mathbf{3}$-$\mathbf{5}$ and 
$\m a$ and $\m b$ are quotients of $\m a\times \m b$,
therefore some relationships must hold
between the type and structure of $\m a$
and the type and structure of $\m b$.
By comparing the possible types and structures of the algebras
$\m a, \m b$, and $\m a\times \m b$ we learn the following:
\smallskip

\begin{enumerate}
\item If $\m a\in {\mathcal V}_{\textrm{fin}}$ is nontrivial and of types
$\mathbf{3}$-$\mathbf{5}$, then since $|A|=2$ we have
that $|A\times A|=4$. Since we also have
$\m a\times \m a\in {\mathcal V}_{\textrm{fin}}$, we derive
that, since $\m a\times \m a$ has more than two elements,
it must be of
type $\mathbf{1}$ or $\mathbf{2}$. But $\m a\times \m a$ will have type
$\mathbf{1}$ iff $\m a$ does and $\m a\times \m a$ will have type
$\mathbf{2}$ iff $\m a$ does. So, it is impossible for $\m a\times\m a$
to have types $\mathbf{3}$-$\mathbf{5}$ (nonabelian types) and also
type $\mathbf{1}$ or type $\mathbf{2}$ (abelian types).
The conclusion is that ${\mathcal V}_{\textrm{fin}}$
cannot contain any algebras of types $\mathbf{3}$-$\mathbf{5}$.
\smallskip

\item If ${\mathcal V}_{\textrm{fin}}$ contains a nontrivial algebra $\m a$
  of type {\bf 1} and a nontrivial algebra $\m b$
  of type {\bf 2}, then $\m a\times \m b\in {\mathcal V}_{\textrm{fin}}$
  will not be of any of the types we have enumerated.
  Thus
  ${\mathcal V}_{\textrm{fin}}$
  consists exclusively of algebras of type {\bf 1}
  or exclusively of algebras of type {\bf 2}.
\smallskip

\item
Assume that ${\mathcal V}_{\textrm{fin}}$
consists exclusively of algebras of type {\bf 2},
and that $\m a, \m b\in {\mathcal V}_{\textrm{fin}}$ are nontrivial.
Since $\m a\times \m b\in {\mathcal V}_{\textrm{fin}}$
must be of type {\bf 2},
$\m a\times \m b$ is polynomially equivalent to an
$\mathbb F$-vector space for some finite field $\mathbb{F}$,
so both quotients $\m a$ and $\m b$ of
$\m a\times \m b$ must be polynomially equivalent to
vector spaces over the same finite field. Thus, in type {\bf 2},
there is a single finite field $\mathbb F$ such that
all members of ${\mathcal V}_{\textrm{fin}}$
are polynomially equivalent to vector spaces over $\mathbb F$.
\smallskip

\item
By a similar argument, if all members of
${\mathcal V}_{\textrm{fin}}$ are of type {\bf 1},
then there is a single group $G$
such that all members of 
${\mathcal V}_{\textrm{fin}}$ are polynomially equivalent
to $G$-sets. In this case, $\mathcal V$ is an essentially unary
variety so an upper bound on the size of the group $G$ is
$|\m f_{\mathcal V}(1)|$, which is finite. Thus,
when all members of
${\mathcal V}_{\textrm{fin}}$ are of type {\bf 1},    
there is a single finite group $G$ such that
all members of ${\mathcal V}_{\textrm{fin}}$ are
polynomially equivalent
  to $G$-sets. 
\end{enumerate}  
\smallskip

The previous observations explain why, if $\mathcal{V}$
is a locally finite, $\langle 0,1\rangle$-minimal variety,
then either there is a finite group $G$ such that
all finite algebras in $\mathcal{V}$ are polynomially equivalent
to $G$-sets or there is a finite field $\mathbb{F}$
such that
all finite algebras in $\mathcal{V}$ are polynomially equivalent
to $\mathbb{F}$-vector spaces. The next lemma extends
these conclusions to the infinite members
of $\mathcal{V}$, as well.

\begin{lm} \label{infinite}
Let $\mathcal{V}$ be a locally finite variety.  
\begin{enumerate}
\item If there is a finite group $G$ such that
  every member of $\mathcal{V}_{\textrm{fin}}$ is polynomially
  equivalent to a $G$-set, then every member of $\mathcal{V}$
is polynomially
  equivalent to a $G$-set. 
\item If there is a finite field $\mathbb{F}$ such that
  every member of $\mathcal{V}_{\textrm{fin}}$ is polynomially
  equivalent to an $\mathbb{F}$-vector space,
  then every member of $\mathcal{V}$
is polynomially
  equivalent to an $\mathbb{F}$-vector space.
\end{enumerate}
\end{lm}

\begin{proof}
For Item~(1), the hypothesis that   
there is a finite group $G$ such that
  every member of $\mathcal{V}_{\textrm{fin}}$ is polynomially
  equivalent to a $G$-set implies that every
$\mathcal{V}$-term depends on at most one variable
  on each algebra of the
  class $\mathcal{V}_{\textrm{fin}}$.
  This class contains the finitely generated free
  algebras of $\mathcal{V}$, hence it follows that
every $\mathcal{V}$-term depends on at most one variable
on any algebra in $\mathcal{V}$.
By changing the language, if necessary, we may
assume that $\mathcal{V}$ is a unary variety and
that there is one basic operation
$u(x)$ for each element $u$ of the finite
algebra $\m f_{\mathcal{V}}(x)$.
Any operation $u(x)$ that does not depend on its variable
on $\m f_{\mathcal{V}}(x)$,
must be a constant of $\m f_{\mathcal{V}}(x)$,
hence will satisfy $u(x)\approx u(y)$ in
$\m f_{\mathcal{V}}(x)$, hence will satisfy $u(x)\approx u(y)$
in $\mathcal{V}$.
Since $\m f_{\mathcal{V}}(x)$ is polynomially equivalent
to a $G$-set, any $u(x)$ that depends on its variable
on $\m f_{\mathcal{V}}(x)$,
must be a permutation $u(x)=u_g(x)$
of the universe of $\m f_{\mathcal{V}}(x)$
that arises from the action of an element $g$ of $G$.
Hence, for the identity element $1$ of $G$ and for
any elements $g,h\in G$ the identities
\begin{equation*}
  u_1(x)\approx x\quad\text{and}\quad u_g(u_h(x))\approx u_{gh}(x)
\end{equation*}
will be satisfied on
the $1$-generated free algebra
$\m f_{\mathcal{V}}(x)$, and since these are $1$-variable
identities they will be satisfied throughout the variety
$\mathcal{V}$.
Altogether, we see that $\mathcal{V}$
is term equivalent to a variety
whose operations
are unary and consist of a finite set of constant functions
and an action of $G$ 
on any algebra in $\mathcal{V}$.
Thus, every algebra in $\mathcal{V}$
is polynomially equivalent to a $G$-set.

For Item~(2), the hypothesis that   
there is a finite field $\mathbb{F}$ such that
  every member of $\mathcal{V}_{\textrm{fin}}$ is polynomially
  equivalent to an $\mathbb{F}$-vector space
  implies that the finite algebra
  $\m f_{\mathcal{V}}(x,y,z)$
  is polynomially
  equivalent to an $\mathbb{F}$-vector space.
  This algebra will have a ternary Maltsev
  term operation $x-y+z$ (see Exercise~2.8 of \cite{clones}),
  and any Maltsev term operation for
  $\m f_{\mathcal{V}}(x,y,z)$ will be a
  Maltsev term operation for $\mathcal{V}$.
  Any failure of the term condition in any
  algebra $\m a\in \mathcal{V}$ involves finitely
  many elements of $\m a$, hence will be a failure
  of the term condition in some finitely generated
  subalgebra $\m b\leq \m a$. But since
  $\m b\in \mathcal{V}_{\textrm{fin}}$,
  $\m b$ must be polynomially equivalent to a vector
  space, hence will not fail the term condition.
  It follows that every algebra in $\mathcal{V}$
  satisfies the term condition,
  hence $\mathcal{V}$ is an abelian Maltsev variety.
  The main result of \cite{herrmann} implies that 
  such varieties are affine.
  This means that there
  is some ring $\m r$ such that every
  algebra in $\mathcal{V}$ is polynomially
  equivalent to a left $\m r$-module, and
  $\m r$ acts faithfully on $\m f_{\mathcal{V}}(x,y)$.
  The ring $\m r$ is completely determined
  by the algebra structure on 
  $\m f_{\mathcal{V}}(x,y)$ in the following sense:
  if $\m f_{\mathcal{V}}(x,y)$ is polynomially
  equivalent to a faithful
  left $\m r$-module and also to a faithful left $\m s$-module,
  then $\m r\cong \m s$.
Summarizing, we have that
  every member of $\mathcal{V}_{\textrm{fin}}$ is polynomially
  equivalent to an $\mathbb{F}$-vector space,
  every member of $\mathcal{V}$ is polynomially
  equivalent to an $\m r$-module, and
  these classes overlap nontrivially in algebras
  where both $\mathbb{F}$ and $\m r$
  act faithfully. By the uniqueness of the
ring of an affine variety, we must have $\m r\cong \mathbb{F}$.
\end{proof}

For our main theorem, Theorem~\ref{main_thm}, we will want to know
the structure of the subalgebras of free algebras
in locally finite, $\langle 0,1\rangle$-minimal varieties.
This information is contained
in the next two lemmas.

\begin{lm} \label{Gset_lm}
  \emph{(The type {\bf 1} case.)}
Let $\mathcal V$ be a nontrivial locally finite variety
satisfying the hypothesis that there is some
finite group $G$ such that all members of $\mathcal{V}$ are
polynomially equivalent to $G$-sets.
If $\mathcal V$ has a pseudoconstant, then
assume that $\mathcal V$ also has a
constant.

Let $\m C_{\m f}$ be the subuniverse of
$\m f_{\mathcal V}(\kappa)$ consisting of the true
constants of $\m f_{\mathcal{V}}(\kappa)$.
The term reduct of $\m f_{\mathcal V}(\kappa)$
to the language of $G$-sets 
is a disjoint union of $G$-subalgebras
\[\m c_{\m f}\sqcup \m g(\kappa)\]
where $\m g(\kappa)$ is a free $G$-set of rank $\kappa$.
Any $\mathcal V$-subalgebra 
$\m s\leq \m f_{\mathcal V}(\kappa)$ has the form
$\m c_{\m f}\sqcup {\m s}'$ for some
$G$-subalgebra ${\m s}'$ of $\m g(\kappa)$.

In particular, $\mathcal V$ is a Schreier variety.
\end{lm}  

\begin{proof}
As noted in the first paragraph
of the proof of Lemma~\ref{infinite},
$\mathcal{V}$ is an essentially unary variety.
Also, if $\mathcal{V}$ has a unary constant operation, $u(x)$,
then our hypothesis about pseudoconstants
implies that there is a constant $c$ such that
$u^{\m a}(A) = \{c^{\m a}\}$ for every $\m a\in \mathcal{V}$.
Thus, by changing to a
term equivalent variety if necessary,
we may assume that the fundamental operations of $\mathcal{V}$
consist of a group $G = \{1, g_2, \ldots,g_n\}$
of unary operations and a set $C = \{c_1,\ldots,c_k\}$
of zeroary operations, and we may assume that all elements
of $G\cup C$ interpret as distinct operations in any generic
algebra of $\mathcal{V}$. Notice that $G\cup C$
may be considered to be
a representative set of unary term operations
for $\mathcal{V}$, provided we treat any $c\in C$
as also representing a unary pseudoconstant $u(x)$
for which $u^{\m a}(x)=c^{\m a}$ for any $\m a\in \mathcal{V}$. 

Let $X_{\kappa} = \{x_{\alpha}\;|\;\alpha<\kappa\}$
be a set of free generators for $\m f_{\mathcal{V}}(\kappa)$.
Since $\m f_{\mathcal{V}}(\kappa)$ is generated by 
$X_{\kappa}$ under the unary term operations, we get that
$\m f = \m f_{\mathcal{V}}(\kappa)$ is a union
$F=\{c^{\m f}\;|\;c\in C\}\cup \bigcup_{\alpha<\kappa} G\cdot x_{\alpha}$
where the first summand consists of the interpretations of the constants
in $\m f$ and each succeeding summand is the $G$-orbit
of some free generator $x_{\alpha}\in X_{\kappa}$.

\begin{center}
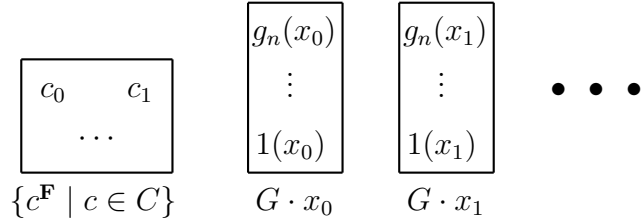

\setlength{\unitlength}{.5mm}
\begin{picture}(160,60)
\thicklines
\put(0,10){\line(1,0){40}}
\put(0,10){\line(0,1){30}}
\put(40,40){\line(-1,0){40}}
\put(40,40){\line(0,-1){30}}

\put(60,10){\line(1,0){25}}
\put(60,10){\line(0,1){45}}
\put(85,55){\line(-1,0){25}}
\put(85,55){\line(0,-1){45}}

\put(100,10){\line(1,0){25}}
\put(100,10){\line(0,1){45}}
\put(125,55){\line(-1,0){25}}
\put(125,55){\line(0,-1){45}}

\put(140,30){$\bullet$}
\put(150,30){$\bullet$}
\put(160,30){$\bullet$}

\put(-3,0){$\{c^{\m F} \mid c\in C\}$}
%\put(15,0){$\m C_{\m f}$}
\put(62,0){$G\cdot x_0$}
\put(102,0){$G\cdot x_1$}

\put(62,15){$1(x_0)$}
\put(70,30){$\vdots$}
\put(61,45){$g_n(x_0)$}

\put(102,15){$1(x_1)$}
\put(110,30){$\vdots$}
\put(101,45){$g_n(x_1)$}

\put(5,30){$c_0$}
\put(28,30){$c_1$}
\put(15,17){$\cdots$}
\end{picture}
\medskip

\captionof{figure}{The algebra $\m f = \m f_{\mathcal{V}}(\kappa)$} \label{fig1}
\end{center}

It is not hard to check that
\begin{itemize}
\item Each constant
  operation
  maps all elements of
  $F=\{c^{\m f}\;|\;c\in C\}\cup \bigcup_{\alpha<\kappa} G\cdot x_{\alpha}$
  into the first summand, $\{c^{\m f}\;|\;c\in C\}$.
\item Each summand,    $\{c^{\m f}\;|\;c\in C\}$ or $G\cdot x_{\alpha}$,
  is closed under the action of the group $G$.
\item Summands containing a free generator are regular $G$-orbits.
\item The expression
  \[F=\{c^{\m f}\;|\;c\in C\}\cup \bigcup_{\alpha<\kappa} G\cdot x_{\alpha}
=\{c^{\m f}\;|\;c\in C\}\sqcup \bigsqcup_{\alpha<\kappa} G\cdot x_{\alpha}
  \]
  is a disjoint union decomposition.
\end{itemize}
We discuss only the last bullet point.
We must explain why, if $x_{\alpha}\neq x_{\beta}\in X_{\kappa}$, then
$G\cdot x_{\alpha}$ is disjoint
from both (i) $\{c^{\m f}\;|\;c\in C\}$ and (ii) $G\cdot x_{\beta}$.
For (i), if $G\cdot x_{\alpha} \cap \{c^{\m f}\;|\;c\in C\}$
is nonempty, then for some $g_i\in G$ we have
$g_i^{\m f}(x_{\alpha})=c^{\m f}$ for some $c\in C$.
This leads to $\mathcal{V}\models g_i(x_{\alpha})\approx c$,
hence $g_i$ is a pseudoconstant of $\mathcal{V}$.
This contradicts the fact that every unary term operation
in $G$ is a permutation of finite order throughout $\mathcal{V}$.
For (ii), if $G\cdot x_{\alpha} \cap G\cdot x_{\beta}$ is nonempty for some
$x_{\alpha}\neq x_{\beta}$, then there exist $g_i, g_j\in G$ such that
$g_i^{\m f}(x_{\alpha})=g_j^{\m f}(x_{\beta})\in G\cdot x_{\alpha} \cap G\cdot x_{\beta}$.
Since $x_{\alpha}$ and $x_{\beta}$ are free generators of $\m f$,
this leads to  $\mathcal{V}\models g_i(x_{\alpha})\approx g_j(x_{\beta})$,
or $\mathcal{V}\models g_j^{-1}g_i(x_{\alpha})\approx x_{\beta}$.
From this and variable specialization, we deduce that
$\mathcal{V}\models g_j^{-1}g_i(x_{\alpha})\approx x_{\alpha}$.
Then, we deduce that
$\mathcal{V}\models x_{\alpha}\approx g_j^{-1}g_i(x_{\alpha})\approx x_{\beta}$,
hence $\mathcal{V}\models x_{\alpha}\approx x_{\beta}$.
Since $x_{\alpha}\neq x_{\beta}$, this forces
$\mathcal{V}$ to be a trivial variety, which
is contrary to one of the assumptions
of this lemma.

The first two bullet points of the preceding paragraph
imply that the set $\{c^{\m f}\;|\;c\in C\}$ of true
constants in $\m f$ is a subuniverse.
The subalgebra supported by this set is the subalgebra we called
$\m c_{\m f}$ in the statement of the lemma.
The fourth bullet point of the preceding paragraph
implies that, for $\m g(\kappa):=\bigcup_{\alpha<\kappa} G\cdot x_{\alpha}$,
we have that
$\m f = \m c_{\m f}\sqcup \m g(\kappa)$
is a disjoint union. The second bullet point of the preceding paragraph
implies that both $\m c_{\m f}$ and $\m g(\kappa)$
are $G$-subalgebras of $\m f$. The third
bullet point of the preceding paragraph suffices to imply that
$\m g(\kappa)$
is a free $G$-set over
$X_{\kappa}$.

Let $\m s\leq \m f$ be a subalgebra. $\m s$ must be closed under
the constant operations of $\mathcal{V}$,
so $\m c_{\m f}\leq \m s$. If $\m s$ intersects
some orbit $G\cdot x_{\alpha}$ of $\m g(\kappa)$,
then $\m s$ must contain the entire orbit
since any element of the orbit generates
all other elements under $G$.
This means that $S\cap G\cdot x_{\alpha}\neq \emptyset$ iff
$x_{\alpha}\in S$. It follows that
$\m s$ is generated as a subalgebra of $\m f$
by the subset $S\cap X_{\kappa}$ of the free generating
set $X_{\kappa}$.
Any subalgebra of a free algebra that
is generated by a subset of the free generating set is free
itself over that subset, so $\m s$ is free in $\mathcal{V}$
over $S\cap X_{\kappa}$. This explains why
a subalgebra of a free algebra in $\m v$ is free,
and also indicates why
$\m s = \m c_{\m f}\sqcup {\m s}'$ for the
$G$-subalgebra ${\m s}'$ of $\m g(\kappa)$ that is generated by
$S\cap X_{\kappa}$. 
\end{proof}

We remark that if $\mathcal{V}$ has no constant
operations in the previous lemma, then
$\mathcal{V}$ is term equivalent to the variety
of $G$-sets for some finite group $G$.

\begin{lm} \label{Vect_lm}
  \emph{(The type {\bf 2} case.)}
Let $\mathcal V$ be a nontrivial locally finite variety
satisfying the hypothesis that there is some
finite field $\mathbb{F}$ such that all members of $\mathcal{V}$ are
polynomially equivalent to $\mathbb{F}$-vector spaces.
If $\mathcal V$ has a pseudoconstant, then
assume that $\mathcal V$ also has a
constant.

\begin{enumerate}
\item \emph{(The case where $\mathcal V$ does not have a pseudoconstant.)}
Let $\mathcal U$ be the idempotent subvariety of $\mathcal V$.
(This is the subvariety of $\mathcal V$
axiomatized relative to $\mathcal V$ by all identities of the form
$u(x)\approx x$ where $u(x)$ is a unary term.)
For any cardinal $\kappa$,
$\m f_{\mathcal V}(\kappa)\cong \m f_{\mathcal U}(\kappa)\times 
\m f_{\mathcal V}(1)$. Moreover, any subalgebra
$\m s\leq \m f_{\mathcal V}(\kappa)$ is isomorphic
to an algebra of the form $\m s'\times \m f_{\mathcal V}(1)$
for some $\m s'\leq \m f_{\mathcal U}(\kappa)$.
\item \emph{(The case where $\mathcal V$ has a constant.)}
Let $\mathcal U$ be the subvariety of $\mathcal V$
defined by all identities of the form
$c\approx d$ whenever $c$ and $d$ are constants of $\mathcal{V}$.
For any cardinal $\kappa$,
$\m f_{\mathcal V}(\kappa)\cong \m f_{\mathcal U}(\kappa)\times 
\m f_{\mathcal V}(0)$. Moreover, any subalgebra
$\m s\leq \m f_{\mathcal V}(\kappa)$ is isomorphic
to an algebra of the form ${\m s}'\times \m f_{\mathcal V}(0)$
for some ${\m s}'\leq \m f_{\mathcal U}(\kappa)$.
\end{enumerate}

In either of Cases~(1) or~(2), $\mathcal V$ is a Schreier variety.
\end{lm}  

The approach that we will take to prove this lemma
will be to compare the variety $\mathcal{V}$
to a minimal subvariety $\mathcal{U}$ of $\mathcal{V}$.
The inclusion $\mathcal{U}\subseteq \mathcal{V}$
realizes $\mathcal{U}$ as a reflective subcategory of $\mathcal{V}$.
Each algebra $\m a\in\mathcal{V}$ will have
a reflection
in $\mathcal{U}$. The $\mathcal{U}$-reflection
of $\m a\in \mathcal{V}$ may be expressed as 
$\m a/\theta$ where $\theta$ is the least congruence
on $\m a$ such that $\m a/\theta\in\mathcal{U}$.
The $\mathcal{U}$-reflection of $\m f_{\mathcal{V}}(\kappa)$
satisfies the universal mapping property that
defines $\m f_{\mathcal{U}}(\kappa)$, hence
the $\mathcal{U}$-reflection of $\m f_{\mathcal{V}}(\kappa)$
is isomorphic to $\m f_{\mathcal{U}}(\kappa)$. Our goal will be to use
the $\mathcal{U}$-reflection arrow
$\m f_{\mathcal{V}}(\kappa)\twoheadrightarrow \m f_{\mathcal{U}}(\kappa)$
to establish
a direct product decomposition of 
$\m f_{\mathcal{V}}(\kappa)$ into the factors 
$\m f_{\mathcal{U}}(\kappa)$ and
a small rank free algebra of $\mathcal{V}$,
$\m f_{\mathcal{V}}(i)$, where $i=1$ in Case~(1)
and $i=0$ in Case~(2).

\begin{center}
\setlength{\unitlength}{.5mm}
\begin{picture}(60,45)
\put(0,0){$\m f_{\mathcal{U}}(\kappa)$}
\put(35,0){$\in$}
\put(55,0){$\mathcal{U}$}

\put(7,22){\rotatebox{270}{$\twoheadrightarrow$}}
\put(55,15){\rotatebox{90}{$\subseteq$}}

%\put(-40,0){$\m f_{\mathcal{V}}(\kappa)/\theta$}
\put(0,30){$\m f_{\mathcal{V}}(\kappa)$}
\put(35,30){$\in$}
\put(55,30){$\mathcal{V}$}
\end{picture}
\medskip

\captionof{figure}{The $\mathcal{U}$-reflection of
$\m f_{\mathcal{V}}(\kappa)$ is $\m f_{\mathcal{U}}(\kappa)$} \label{fig2}
\end{center}

There are two
Cases in the statement of this lemma,
and the arguments justifying them follow
the same pattern.
We briefly describe the pattern before
developing the arguments in each of the two cases.
In this description we will assume that
\begin{equation}\label{nondegenerate}
\text{$\kappa>0$, and if $\mathcal{V}$
is an idempotent variety,
then $\kappa>1$.}
\end{equation}
Notice that the case ``$\mathcal{V}$ is an idempotent variety''
can only occur in Case~(1).

Let us say a few words now about the cases
when~\eqref{nondegenerate} fails,
so that our later argument will be complete.
It is clear from the assumptions made in the statement of the lemma 
that in Case~(1) we have 
$\m f_{\mathcal{V}}(\kappa) = \emptyset = \m f_{\mathcal{U}}(\kappa)$
if $\kappa=0$ and
$|\m f_{\mathcal{V}}(\kappa)| = |\m f_{\mathcal{U}}(\kappa)| = 1$ if
$\mathcal{V}$ is idempotent and $\kappa=1$,
while in Case~(2) we have
$|\m f_{\mathcal{U}}(\kappa)| = 1$ if $\kappa=0$.
This is enough information to establish
the desired direct product decomposition of free algebras
$\m f_{\mathcal V}(\kappa)\cong\m f_{\mathcal U}(\kappa)\times \m f_{\mathcal V}(i)$
where $i=1$ in Case~(1) and $i=0$ in Case~(2).
If~\eqref{nondegenerate} fails, then
$\m f_{\mathcal{V}}(\kappa)$ has no proper subalgebras
in either of the two Cases.
Hence, in the last claim of Cases~(1) or (2), where we
assume that $\m s\leq \m f_{\mathcal{V}}(\kappa)$,
this assumption reduces to
$\m s = \m f_{\mathcal{V}}(\kappa)$.
Therefore we have $\m s\cong \m s'\times \m f_{\mathcal{V}}(i)$
for $\m s'=\m f_{\mathcal{U}}(\kappa)$ by the direct product
decomposition of free algebras that has already
been established.
This paragraph now contains enough information
to establish the validity of the
lemma in the case when~\eqref{nondegenerate} fails.
\bigskip

Let's assume henceforth that~\eqref{nondegenerate} holds.
This additional assumption on $\kappa$ and $\mathcal{V}$ assures that
$|\m f_{\mathcal{V}}(\kappa)|>1$, so the free algebra $\m f_{\mathcal{V}}(\kappa)$
is nontrivial.
Our arguments under this assumption will follow Stages (A)--(E), described here:
  
\begin{enumerate}  
\item[(A)] We examine the clone of term operations
  of $\m f:=\m f_{\mathcal{V}}(\kappa)$ in order
  to understand the unary component of this clone.
\item[(B)] We show that the linearization
  $\m f_{\nabla}:=\m f^2/\Delta_{1,1}$ of $\m f$
  (defined in Chapter~9 of \cite{FM})
  generates a minimal
  subvariety $\mathcal{U}\subseteq \mathcal{V}$.
  In Case~(1), $\mathcal{U}$ is the idempotent subvariety
  of $\mathcal{V}$. This subvariety is term equivalent
  to the variety of affine spaces over the field $\mathbb{F}$.
  In Case~(2), $\mathcal{U}$
  is the subvariety of $\mathcal{V}$
  axiomatized by identities that
  identify all
  pseudoconstants of $\mathcal{V}$. This subvariety is term equivalent
  to the variety of $\mathbb{F}$-vector spaces.
\item[(C)]
  The minimal subvariety $\mathcal{U}$ is a reflective
  subcategory of $\mathcal{V}$.
There is a smallest
congruence $\theta$ on $\m f$ (namely the kernel of the
$\mathcal{U}$-reflection arrow)
such that
$\m f/\theta\in \mathcal{U}$. We will see that at least
one $\theta$-class supports a subalgebra $\m s$
of $\m f_{\mathcal{V}}(\kappa)$. The algebra $\m s$
is isomorphic to
$\m f_{\mathcal{V}}(1)$ in Case~(1) and to
$\m f_{\mathcal{V}}(0)$ in Case~(2).
This congruence has a permuting complement $\theta'$
that has $\m s$ as a transversal.
\item[(D)]
Any subalgebra
$\m s\leq \m f_{\mathcal V}(\kappa)$ is a union of $\theta$-classes.
\item[(E)]
$\mathcal{V}$ is Schreier.
\end{enumerate}
This will be enough. Stages (A)--(C)
suffice to yield a direct product decomposition
of free algebras
\[
\m f_{\mathcal{V}}(\kappa)\cong 
\m f_{\mathcal{V}}(\kappa)/\theta\times 
\m f_{\mathcal{V}}(\kappa)/\theta'\cong
\m f_{\mathcal{U}}(\kappa)\times \m f_{\mathcal{V}}(i)
\]
where $i=1$ in Case~(1) and $i=0$ in Case~(2).
The existence of this decomposition is
the first claim in each of the two Cases of this lemma.
The second and final claim in both Cases~(1) and (2)
will follow from Stage~(D).
The final statement of the lemma will be explained
by Stage (E).
  
The rest of this proof consists of $10$ components:
for each of the Cases~(1) and (2)  we must explain (A)--(E).

In either case, the $\mathcal{V}$-free algebra
$\m f := \m f_{\mathcal{V}}(\kappa)$ is polynomially equivalent to
some vector space $\m v$ over some finite field $\mathbb{F}$.
Much is known about the structure of
algebras $\m a$ that are polynomially
equivalent to some faithful
left $\m r$-module $\m m$ for some given
unital, associative ring $\m r$.
In particular, $\m a$ and $\m m$ share
the same universe ($A=M$) and the
same polynomial clone ($\textrm{Pol}(\m a)=\textrm{Pol}(\m m)$).
The polynomial operations of a module are well understood,
namely they are the module operations of the form
\[
p(x_1,\ldots,x_k) = \sum_{i=1}^k (\alpha_i\cdot x_i) + m =
\textrm{(linear part) + (constant).}
\]
Here $\alpha_i\in R$ and $m\in M$.
A polynomial of this form will be idempotent when
$m=0$ and $\sum_{i=1}^k \alpha_i = 1$ in $\m r$.
The following are some basic facts about the clone
of term operations of an algebra $\m a$
that is polynomially equivalent to a left $\m r$-module $\m m$.
\begin{itemize}
\item Every idempotent polynomial operation
of $\m m$ is a term operation
of $\m a$, and every term operation of $\m a$ is a polynomial
operation of $\m m$.
\item More generally, a polynomial operation
$p(x_1,\ldots,x_k)=\sum_{i=1}^k (\alpha_i\cdot x_i) + m$ of $\m m$
is a term operation of $\m a$ if and only if
the unary specialization
$\widehat{p}(x):=p(x,x,\ldots,x) = (\sum_{i=1}^k \alpha_i)\cdot x + m$
is a unary term operation of $\m a$.
\item  
The set of pairs $(1-\beta,m)\in R\times M$,
where $\beta x+m$ is a unary term operation of $\m a$,
is an $\m r$-submodule of the $\m r$-module
${}_{\m r}\m r\times \m m$.
\item  
Conversely, for every $\m r$-submodule
$\m n\leq {}_{\m r}\m r\times \m m$,
there is a uniquely determined clone of an algebra
on the set $M$ that is polynomially equivalent to $\m m$
and has unary component
\[
\{\beta x + m\;|\; (1-\beta,m)\in N\}.
\]
\end{itemize}
From the first two bullet points, it follows
that the clone of term operations of $\m a$
is generated by its idempotent subclone
and the unary component of its clone.
The idempotent subclone of $\m a$ agrees with that of $\m m$.
From the third
and fourth bullet points, it is clear that the data
of the unary component of the clone of term
operations of $\m a$ is encoded in a submodule
of the $\m r$-module
  ${}_{\m r}\m r\times \m m$.
These results can be derived from Chapter 2 of
\cite{clones}, especially Proposition~2.6
and its proof.

In our situation, the ring involved is the
finite field $\m r = \mathbb{F}$
and the $\m r$-module $\m m$ is the $\mathbb{F}$-vector space
$\m v$. Suppose that $\m n\leq {}_{\mathbb{F}}\mathbb{F}\times \m v$
is the
$\mathbb{F}$-subspace
that determines the unary component of the
clone of term operations of $\m f=\m f_{\mathcal{V}}(\kappa)$.
With these notations, we present the proofs of
Cases~(1) and~(2) of Lemma~\ref{Vect_lm}.

\begin{proof}[Proof of Lemma~\ref{Vect_lm}, Case~(1), when $\kappa>0$]   
(A): Since we are in Case~(1), we know that
$\mathcal{V}$ has no pseudoconstants.
It follows that $\m f_{\mathcal{V}}(\kappa)$
has no constant unary term operations.
In other words, there is no unary term operation
of this algebra whose vector space representation is
$\beta x+v$ with $\beta=0$.
Thus, $\m n$ can contain no pair
$(1-\beta,v)$ where $\beta=0$.
But if $(\alpha,u)\in N$ for some $\alpha\in \mathbb{F}-\{0\}$,
then $\alpha^{-1}(\alpha,u)=(1,\alpha^{-1}u)\in N$ and
$(1,\alpha^{-1}u)$ has the form $(1-\beta,v)$
for $\beta=0$. This cannot happen
if $\m f$ has no pseudoconstants,
so we cannot have $(\alpha,u)\in N$ if $\alpha\neq 0$.
This implies that 
$\m n\leq \{0\}\times \m v$. From this, we learn that the
unary term operations of $\m f$
have the form $\beta x+v$ only when $\beta=1$,
which means that $\beta x+v=x+v$ is a translation.
Since $\m n$ is an $\mathbb{F}$-subspace
of ${}_{\mathbb{F}}\mathbb{F}\times \m v$
the unary component of the clone of term
operations of $\m f$ must be the set of all translations 
$\{x+v\;|\;v\in W\}$
from some $\mathbb{F}$-subspace $W\leq \m V$.
This completes Stage~(A) for Case~(1).

(B): Let's use the result from (A) to complete Stage~(B).
$\m f$ is an abelian algebra in a Maltsev variety,
so $\m f^2$ has a unique congruence $\Delta_{1,1}$
that has the diagonal of $\m f^2$ as a class.
The quotient $\m f_{\nabla}:=\m f^2/\Delta_{1,1}$
is called the {\bf linearization} of $\m f$.
It has the same cardinality as $\m f$,
so it will be nontrivial, since $\m f$ is.
The linearization of $\m f$ belongs
to the variety $\mathcal{V}$, because $\m f\in \mathcal{V}$
and $\m f_{\nabla}$ is
constructed from $\m f$ using only constructions of the form
$\mathsf{H}$, 
$\mathsf{S}$, or
$\mathsf{P}$.
Hence, for us, $\m f_{\nabla}$ is also polynomially
equivalent to an $\mathbb{F}$-vector space.
In particular, the clone of term operations
of $\m f_{\nabla}$ is generated by
its idempotent vector space polynomials
and by some unary operations.
The unary component of the clone of term operations
$\m f_{\nabla}$ is induced by that of $\m f$.
But, it turns out, vector space
translations do not survive linearization.
Each translation $x+v$ on $\m f$
induces the identity function on 
$\m f_{\nabla}$. This implies that the
clone of term operations of
$\m f_{\nabla}$ is generated entirely
by idempotent operations,
so $\m f_{\nabla}$ must be an idempotent algebra.
One justifies the claim that
`translations do not survive linearization'
as follows. Choose any element
$(a,b)/\Delta_{1,1}\in \m f_{\nabla}$
and any unary term operation $u(x)$
of this algebra. In $\m f$, we have seen (by (A))
that the vector space representation of the
operation $u(x)$ is some translation $u(x)=x+v$.
Hence, in
$\m f_{\nabla}$, $u((a,b)/\Delta_{1,1})=(a+v,b+v)/\Delta_{1,1}$.
But $(a+v,b+v)/\Delta_{1,1} = (a,b)/\Delta_{1,1}$,
since the term condition guarantees that
\[
(a,b)^t=
\begin{bmatrix}a\\
  b
\end{bmatrix}
=
\begin{bmatrix}a+\underline{0}  \\
  b+\underline{0}
\end{bmatrix}
\quad
\Delta_{1,1}
\quad 
\begin{bmatrix}a+\underline{v}  \\
  b+\underline{v}
\end{bmatrix}
=\begin{bmatrix}a+v\\
  b+v
\end{bmatrix} = (a+v,b+v)^t.
\]
Thus, the arbitrarily chosen element
$(a,b)/\Delta_{1,1}\in \m f_{\nabla}$
is a fixed point of the arbitrarily chosen
unary term operation $u(x)$.
This is enough to show that
$\m f_{\nabla}$ is an idempotent algebra.
This completes our work on (B), since our
goal was only to show that the linearization
$\m f_{\nabla}$ is nontrivial and idempotent.
(The fact that a nontrivial idempotent algebra
that is polynomially equivalent to a vector
space generates a minimal variety --- called the variety
of ``affine $\mathbb{F}$-spaces'' --- is well-known,
see \cite{k-k-sz}.)

(C): Next we discuss Stage~(C) in Case~(1).
We will need the fact that every nonempty subuniverse $S$ of
$\m f = \m f_{\mathcal{V}}(\kappa)$ is the congruence class of a unique
congruence $\sigma_S$ of $\m f$. This can be justified as follows.
Since $\m f$
is polynomially equivalent to an $\mathbb{F}$-vector space $\m v$,
it follows from our discussion preceding the proof of Lemma~\ref{Vect_lm}
that the affine $\mathbb{F}$-vector space reduct $\m v_{\text{\rm aff}}$
of $\m v$ (i.e., the idempotent reduct of $\m v$) is a reduct of $\m f$.
Thus, every nonempty subuniverse of $\m f$ supports a
nonempty affine subspace of $\m v_{\text{\rm aff}}$.
The affine space $\m v_{\text{\rm aff}}$ is known to have the property that
every affine subspace is the congruence class of a unique congruence, and
$\m v_{\text{\rm aff}}$ has the same congruences as $\m v$ and $\m f$ (as they
are polynomially equivalent), therefore every nonempty subuniverse $S$ of
$\m f$ is the congruence
class of a unique congruence $\sigma_S$ of $\m f$, as claimed.
Necessarily, $\sigma_S=\cg(S\times S)$, the congruence generated by $S\times S$.

Choose any $s, t\in \m f$.
We claim that if
$\m s=\langle s\rangle_{\m f}$ and 
$\m t=\langle t\rangle_{\m f}$ are
cyclic subalgebras of $\m f$, then $\sigma_S=\sigma_T$.
This can be justified as follows.
The subuniverse $S$ equals the closure of $\{s\}$
under the action of the unary term operations
of $\m f$, which are the unary functions which have
vector space representations that are
translations, $u(x)=x+v$,
$v\in W$, for some fixed subspace $W\leq \m v$.
So, $S=\{s+v\;|\;v\in W\}$ and
similarly, $T=\{t+v\;|\;v\in W\}$.
The generating pairs of $\sigma_S$ and $\sigma_T$
differ by polynomial maps of $\m f$.
That is, a generator $(s+v,s+w)$ of $\sigma_S$
differs from the corresponding generator
$(t+v,t+w)$ of $\sigma_T$
by unary polynomial translations.
That is, the unary polynomial
$p(x)=x-s+t$ of $\m f$,
constructed from the Maltsev operation
$x-y+z$ of $\m f$ and the constants $s$ and $t$,
maps the pair $(s+v,s+w)$ of $\sigma_S$
to the pair $p((s+v,s+w))=(t+v,t+w)$,
which is a generator
for $\sigma_T$. Similarly, the polynomial
$q(x)=x-t+s$
maps the pair $(t+v,t+w)$ of $\sigma_T$
to the pair $(s+v,s+w)$, which is a generator
for $\sigma_S$.
Hence, each one of the congruences $\sigma_S$ and $\sigma_T$
contains pairs that generate the other.
This shows that
\[
\sigma_S = \cg(S\times S) = \cg(T\times T)
= \sigma_T.
\]
This calculation, combined with the fact discussed in the previous paragraph,
shows that there is a congruence $\theta$
on $\m f$ whose classes are exactly the cyclic subuniverses
of $\m f$, which are exactly the 
the orbits of the group $\{x+v\;|\;v\in W\}$
of unary term operations of $\m f$ acting on $F$.
Necessarily, $\theta$ is the least congruence on $\m f$
such that singleton subsets of $\m f/\theta$
are subuniverses; i.e., $\theta$ is the least congruence on $\m f$
such that $\m f/\theta$ is an idempotent algebra.
Since $\mathcal{U}$ is an idempotent subvariety of $\mathcal{V}$,
and $\mathcal{U}$ contains all idempotent algebras
that are polynomially equivalent to $\mathbb{F}$-spaces,
$\theta$ must be the kernel of the
$\mathcal{U}$-reflection arrow
$\m f=\m f_{\mathcal{V}}(\kappa)\twoheadrightarrow\m f_{\mathcal{U}}(\kappa)$
arising from
the inclusion $\mathcal{U}\subseteq \mathcal{V}$. This implies
that $\m f_{\mathcal{V}}(\kappa)/\theta\cong \m f_{\mathcal{U}}(\kappa)$.

The congruences of $\m f$ are the same as the congruences of the
underlying vector space, hence $\Con(\m f)$ is a complemented lattice.
Choose any complement $\theta'$ to $\theta$.
If $z$ is a free generator of $\m f_{\mathcal{V}}(\kappa)$,
then the subuniverse
$Z=\langle z\rangle_{\m f}$ is a $\theta$-class. Since
$\theta'$ is a permuting complement to $\theta$, $Z$ is
a $\theta'$-transversal. Since $Z$ is a subuniverse
of $\m f_{\mathcal{V}}(\kappa)$ that is generated by the free generator
$z$, the subalgebra $\m z$ supported by $Z$
must be freely generated by $z$
and therefore $\m z\cong \m f_{\mathcal{V}}(1)$.
Altogether, these observations yield
$\m f_{\mathcal{V}}(\kappa)/\theta'\cong \m z\cong \m f_{\mathcal{V}}(1)$.
Since $Z$ is a transversal for $\theta'$
we obtain the direct product decomposition
\[
\tag{$\dagger$} \label{ISO}
\m f_{\mathcal{V}}(\kappa)\cong 
\m f_{\mathcal{V}}(\kappa)/\theta\times 
\m f_{\mathcal{V}}(\kappa)/\theta'\cong 
\m f_{\mathcal{U}}(\kappa)\times
\m z
\cong 
\m f_{\mathcal{U}}(\kappa)\times
\m f_{\mathcal{V}}(1).
\]
This concludes the explanation of Stage~(C) in Case~(1)
and also explains the isomorphism in the statement of
Case~(1).

(D): We now turn to Stage~(D) in Case~(1).
Let $\m s\leq \m f_{\mathcal{V}}(\kappa)$ be any subalgebra.
$\m s$ must be closed under the action of the group
of unary term operations of $\m f = \m f_{\mathcal{V}}(\kappa)$.
Since the orbits of this action are the $\theta$-classes, 
$\m s$ is a union of $\theta$-classes.

(E): From the fact proved in Stage (D), that
$\m s$ is a union of $\theta$-classes, we know that
$\m s$ is a product subalgebra
with full second factor in the product
\[
\m f_{\mathcal{V}}(\kappa)\cong 
\m f_{\mathcal{U}}(\kappa)\times
\m f_{\mathcal{V}}(1).
\]
Hence
$\m s$ is isomorphic to $\m s'\times \m f_{\mathcal{V}}(1)$
for some $\m s'\leq \m f_{\mathcal{U}}(\kappa)$.
The idempotent variety $\mathcal{U}$
of ``affine $\mathbb{F}$-spaces''
has the property that every member is free
(see \cite{k-k-sz}).
Thus, $\m s'\cong \m f_{\mathcal{U}}(\lambda)$ for some
$\lambda$. Using the isomorphism in $(\dagger)$, we get
\[
\m s\cong \m s'\times \m f_{\mathcal{V}}(1)\cong 
\m f_{\mathcal{U}}(\lambda)\times
\m f_{\mathcal{V}}(1)\cong
\m f_{\mathcal{V}}(\lambda).
\]
We conclude that every subalgebra $\m s$
of a $\mathcal{V}$-free
algebra $\m f_{\mathcal{V}}(\kappa)$
is $\mathcal{V}$-free. Hence,
$\mathcal{V}$ is a Schreier variety, establishing Stage~(E)
in Case~(1).
\end{proof}

\bigskip

\begin{proof}[Proof of Lemma~\ref{Vect_lm}, Case~(2), when $\kappa>0$]

We follow the same pattern to justify
Stages~(A)--(E) in Case~(2) (when $\mathcal{V}$ has at least one constant)
in the cases where $\kappa>0$.

To simplify the argument, select
one of the constants of $\m f$, call it $\mathbf{0}$,
and treat it as the origin of $\m v$.
Let's argue that there is no loss of generality in
this assumption 
that $\mathbf{0}$ is the origin of $\m v$.
At this point, the only things that have been
assumed about $\m v$ is that it is a vector space,
it has the same universe as $\m f$, and it has the same
polynomial operations as $\m f$.
If we conjugate the clone of $\m v$ by a
vector space polynomial translation, say $t(x)=x+c$,
then we do not change the polynomial clone of $\mathcal{V}$
(it remains the same as the polynomial clone of $\m f$),
but the clone of term operations of $\m v$
will change to an isomorphic clone on the set $V$
where the origin has been translated by $c$.
We may replace the original
vector space $\m v$ by the conjugate
space and preserve all assumptions that we have made about $\m v$,
but move the origin to any  pre-selected vector in $\m v$.
Thus, since $\m f$ has at least one constant,
there is no harm in assuming
that the origin is the interpretation of the
constant $\mathbf{0}$
of $\m f$. Let's continue with the
justification for Stages~(A)--(E) for Case~(2) under this assumption.

(A): The $\mathcal{V}$-free algebra
$\m f := \m f_{\mathcal{V}}(\kappa)$ is polynomially equivalent to
the $\mathbb{F}$-space $\m v$, and the clone
of term operations of $\m f$
contains the clone of term operations of $\m v$.
Let's determine the unary component
of the clone of term operations of $\m f$.
Recall that it is determined by
some $\mathbb{F}$-subspace 
$\m n\leq {}_{\mathbb{F}}\mathbb{F}\times \m v$
in the sense that $u(x) = \beta x+v$
will be the vector space representation
of a unary term operation if and only if
$(1-\beta,v)\in \m n$.
In the previous paragraph we assumed
that the origin $\mathbf{0}$
is named by a constant, hence there is a
unary pseudoconstant $u(x) = \beta x+v = \mathbf{0}$.
That means the pair $(1-\beta,v) =
(1-0,\mathbf{0}) =  (1,\mathbf{0})$
belongs to $\m n$. Since $\m n$ is an $\mathbb{F}$-subspace of 
${}_{\mathbb{F}}\mathbb{F}\times \m v$, it follows
that for every scalar $\alpha\in \mathbb{F}$ we have
\[
\alpha\cdot (1,\mathbf{0}) = (\alpha,\mathbf{0})\in \m n.
\]
That is,
${}_{\mathbb{F}}\mathbb{F}\times \{\mathbf{0}\}\subseteq \m n$.
This implies that $\m n$ is a product subspace
of ${}_{\mathbb{F}}\mathbb{F}\times \m v$,
namely $\m n= {}_{\mathbb{F}}\mathbb{F}\times \m w$ for some
$\mathbb{F}$-subspace $\m w$ of $\m v$.
In particular, $\m n$ is generated by the unary term
operations corresponding to elements of
${}_{\mathbb{F}}\mathbb{F}\times \{\mathbf{0}\}$ ---
these are the scalar multiplications $x\mapsto \alpha x$ ---
together with some unary term
operations corresponding to elements of
$\{0\}\times \m w$ --- these are
unary translations $t(x) = x+v$ for $v$
in $\m w$.
Since the clone of term operations of $\m v$
contains (and is generated by)
the idempotent subclone of $\m f$,
all unary scalar multiplications, and
the constant $\mathbf{0}$, it follows
that the clone of term operations of $\m f$
contains the clone of term operations of $\m v$
and that the clone of term operations of $\m f$
is generated by the clone of term operations of $\m v$
together with the unary translations $t(x) = x+v$ for $v$
in the subspace $\m w$ of $\m v$.
Alternatively, the clone of term operations of $\m f$
is generated by the clone of term operations of $\m v$
together with the constant operations
where the constants come from the subspace
$\m w$, which is the orbit of $\mathbf{0}$ under the translations:
\[\{t(\mathbf{0}) = \mathbf{0} + v = v\;|\;v\in W\} = W.\]
To verify this alternative description of the clone of $\m f$,
observe that if $t(x)=x+v$ is a translation in
the clone of term operations of $\m f$, then
$t(\mathbf{0}) = v\in W$ is a constant in the clone
of term operations of $\m f$. Conversely, if $v$
is a constant in the clone of term operations of
$\m f$, then the unary translation
$t(x) = x-\mathbf{0}+v = x+v$ is a translation in the clone
of term operations of $\m f$, which we constructed
from the Maltsev operation $x-y+z$ of $\m f$ and
the constants $\mathbf{0}$ and $v$ from the clone
of term operations of $\m f$.
Thus, up to term equivalence, $\m f$ may be viewed
as the expansion of $\m v$ by the group of unary
translations $\{x+v\;|\;v\in W\}$, for some subspace $\m w\leq \m v$,
or $\m f$ may be viewed
as the expansion of $\m v$ by the set of constants
from $W$. This completes the argument for Stage (A)
(= understanding the unary component of the clone
of $\m f$).

(B): Let's adhere to the path developed for the
Case~(1) argument by examining
the linearization $\m f_{\nabla}$
of $\m f = \m f_{\mathcal{V}}(\kappa)$.
The linearization of $\m f$ is the same cardinality
as $\m f$, so $\m f_{\nabla}$ is nontrivial, since $\m f$ is.
The linearization satisfies all constant identities
$c\approx d$, since
\[
c^{\m f_{\nabla}} = (c^{\m f},c^{\m f})/\Delta_{1,1}, \quad
d^{\m f_{\nabla}} = (d^{\m f},d^{\m f})/\Delta_{1,1}, 
\]
and $((c^{\m f},c^{\m f}),(d^{\m f},d^{\m f})\in\Delta_{1,1}$
since the diagonal of $\m f^2$ is a $\Delta_{1,1}$-class.
Hence, since $\m f$ is nontrivial, then
$\m f_{\nabla}$ is a nontrivial member
of the subvariety $\mathcal{U}$
that is axiomatized relative to $\mathcal{V}$
by all constant identities.
The clone of term operations of
$\m f_{\nabla}$ is generated by the $\mathbb{F}$-vector space
operations and the constants it contains.
Since all constants were collapsed
to $\mathbf{0}$ during linearization,
and $\mathbf{0}$ is a vector space constant,
we derive that $\m f_{\nabla}$ is term equivalent
to a nontrivial $\mathbb{F}$-vector space.
This shows that the (minimal) variety $\mathcal{U}$ of
$\mathbb{F}$-vector spaces is a nontrivial
subvariety of $\mathcal{V}$.
This completes the argument for Stage~(B) in Case~(2).

(C): Let's discuss Stage~(C) in Case~(2).
The $\emptyset$-generated subuniverse of
$\m f_{\mathcal{V}}(\kappa)$ equals the set $W$ of
the constants of this algebra.
The subalgebra supported
by this subuniverse
is isomorphic to $\m f_{\mathcal{V}}(0)$.
The congruence $\theta = \cg(W\times W)$
is the least congruence that identifies
all constants with $\mathbf{0}$,
hence it must be the kernel of the
$\mathcal{U}$-reflection arrow.
As in the Case~(1) argument,
$\theta$ has a permuting complement $\theta'$
and $W$ is a transversal for $\theta'$.
This allows us to derive that
\[
\m f_{\mathcal{V}}(\kappa) \cong 
\m f_{\mathcal{V}}(\kappa)/\theta\times  
\m f_{\mathcal{V}}(\kappa)/\theta'\cong 
\m f_{\mathcal{U}}(\kappa) \times
\m f_{\mathcal{V}}(0),
\]
as required by the statement of Case~(2).

(D): Now we complete Stage~(D) in Case~(2).
Any subalgebra $\m s\leq \m f_{\mathcal{V}}(\kappa)$
must contain the constants.
$\m s$ must be closed under the unary term
operations of the form $\{t(x) = x+v\;|\;v\in W\}$.
These term operations form a group acting
on $\m f_{\mathcal{V}}(\kappa)$ whose orbits 
are the cosets of $\theta$. Thus, $\m s$
is a union of $\theta$-classes.

(E): We complete Stage~(E) in Case~(2) the same way we completed
it in Case~(1). The fact that any subalgebra
$\m s\leq \m f_{\mathcal{V}}(\kappa)$
is a union of $\theta$-classes implies that
$\m s$ is a product subalgebra
with full second factor in the product
\[
\m f_{\mathcal{V}}(\kappa)\cong 
\m f_{\mathcal{U}}(\kappa)\times
\m f_{\mathcal{V}}(0).
\]
Hence
$\m s$ is isomorphic to $\m s'\times \m f_{\mathcal{V}}(0)$
for some $\m s'\leq \m f_{\mathcal{U}}(\kappa)$.
But $\m f_{\mathcal{U}}(\kappa)$ is term equivalent
to a vector space, and every vector space is
free, so $\m s'\cong \m f_{\mathcal{U}}(\lambda)$ for some
$\lambda$. This shows that 
\[
\m s\cong \m s'\times \m f_{\mathcal{V}}(0)\cong 
\m f_{\mathcal{U}}(\lambda)\times
\m f_{\mathcal{V}}(0)\cong
\m f_{\mathcal{V}}(\lambda).
\]
We conclude that every subalgebra $\m s$
of a $\mathcal{V}$-free
algebra $\m f_{\mathcal{V}}(\kappa)$
is $\mathcal{V}$-free. Hence,
$\mathcal{V}$ is a Schreier variety in Case~(2).
\end{proof}

\section{The Main Theorem} \label{main}

In this section, we prove the main theorem of this paper.

\begin{thm} \label{main_thm}
  Assume that $\mathcal V$ is a nontrivial
  locally finite variety.
  If $\mathcal V$ has a pseudoconstant, then
  assume that $\mathcal V$ also has a
  constant. 
The following are equivalent.
\begin{enumerate}
\item $\mathcal V$ is a Schreier variety.
  \smallskip
  
\item $\mathcal V$ is a Nielsen variety.
  \smallskip

\item Every finite
  algebra in $\mathcal V$ is $\langle 0,1\rangle$-minimal.
  \smallskip

\item $\mathcal V$ is generated by a single
  finite, abelian, $\langle 0,1\rangle$-minimal
  algebra.
\end{enumerate}
\end{thm}

\begin{proof}

[(1)$\Rightarrow$(2)] If every subalgebra of a free algebra
is free, then every finitely generated subalgebra
of a free algebra is free.

[(2)$\Rightarrow$(3)] 
Let $\mathcal V$ be a nontrivial, locally finite,
Nielsen variety and assume $n>0$.

\begin{clm} \label{main_clm}  
If
$t(x_1,\ldots,x_n)$ is an $n$-ary $\mathcal V$-term
that depends on its first variable in $\mathcal V$, then
there is a $\mathcal V$-term
$s(x_1,\ldots,x_n)$ such that
$$
\mathcal V\models s(t(x_1,x_2,\ldots,x_n),x_2,\ldots,x_n)\approx x_1.
$$
\end{clm}

\cproof
Let $\m f=\m f_{\mathcal V}(x_1,\ldots,x_n)=\m f_{\mathcal V}(n)$ be the
$\mathcal V$-free
algebra generated by the finite set $X_n:=\{x_1,\ldots,x_n\}$.
Necessarily $t=t^{\m f}(x_1,\ldots,x_n)\in F$.
Let $\m s\leq \m f$
be the subalgebra of $\m f$
generated by $X_{n-1}:=\{x_2,\ldots,x_n\}$.
Since $\m s$ is generated by the subset $X_{n-1}$
of the free generating set $X_n$ for $\m f$,
$\m s$ is free in $\mathcal V$ over $X_{n-1}$.
The fact that the term $t(x_1,\ldots,x_n)$ depends
on its first variable in $\mathcal V$
is equivalent to the statement
that $t\notin S$. Since $t\notin S=\langle X_{n-1}\rangle$,
the subalgebra $\m t\leq \m f$
generated by $\{t,x_2,\ldots,x_n\}$
properly extends $\m s$, so
\[
|\m f_{\mathcal V}(n-1)|=|\m s|<|\m t|\leq |\m F| = 
|\m f_{\mathcal V}(n)|.
\]
Since $\m t$ is a finitely generated subalgebra
of the $\mathcal V$-free algebra $\m f$, 
$\m t$ is free in $\mathcal V$.
Since the sizes of the free algebras
in a nontrivial locally finite variety
strictly increase as the rank of the algebra
increases, and $\m t$ is a $\mathcal{V}$-free algebra satisfying
\[
|\m f_{\mathcal V}(n-1)|<|\m t|\leq 
|\m f_{\mathcal V}(n)|,
\]
we derive that 
$|\m t| = |\m f|$.
Since $\m t\leq \m f$, we derive that $\m t=\m f$.
Hence $x_1\in \m f$ belongs to the subalgebra $\m t$ of $\m f$
generated
by $\{t,x_2,\ldots,x_n\}$. This fact yields a term
$s$ such that $s^{\m f}(t,x_2,\ldots,x_n)=x_1$.
The relation $s^{\m f}(t^{\m f}(x_1,x_2\ldots,x_n),x_2,\ldots,x_n)=x_1$,
which holds among the free generators
of $\m f$, implies that the identity
$$s(t(x_1,x_2,\ldots,x_n),x_2,\ldots,x_n)\approx x_1$$
holds throughout $\mathcal V$.
\cqed
\bigskip

Next we explain why Claim~\ref{main_clm} implies
that every finite algebra in $\mathcal V$
is $\langle 0,1\rangle$-minimal.
Choose $\m a\in {\mathcal V}_{\textrm{fin}}$
(= the class of finite members of $\mathcal V$).
If $p(x)$ is a nonconstant unary polynomial of $\m a$,
then there exists a term $t(x_1,\ldots,x_n)$ which
depends on its first variable in $\mathcal V$ and a
tuple $\wec{a}=(a_2,\ldots,a_n)\in A^{n-1}$ such that
$p(x) = t^{\m a}(x,\wec{a})$. By Claim~\ref{main_clm},
there exists a term $s(x_1,\ldots,x_n)$ such that
$$
\mathcal V\models s(t(x_1,x_2,\ldots,x_n),x_2,\ldots,x_n)\approx x_1.
$$
From this we get
$s^{\m a}(t^{\m a}(x,\wec{a}),\wec{a})=x$, so
$p(x)= t^{\m a}(x,\wec{a})$ is a
left-invertible self-map of the finite set $A$.
This forces $p(x)$ to be a permutation of $A$,
proving that every nonconstant unary polynomial of $\m a$
is a permutation of $A$.

[(3)$\Rightarrow$(1)]
Assume that every finite algebra in $\mathcal V$ is
$\langle 0,1\rangle$-minimal.
According to the discussion at the beginning
of Section~\ref{decomp} and the results
of Lemmas~\ref{Gset_lm} and \ref{Vect_lm},
$\mathcal V$ is a Schreier variety.

[(3)$\Leftrightarrow$(4)]
We have already proved that (1), (2), and (3) are equivalent,
and now we argue that (4) is equivalent to them.
We first prove that (3)$\Rightarrow$(4).

\begin{clm} \label{3implies4}
Assume that $\mathcal{V}$ satisfies the hypotheses of Theorem~\ref{main_thm}.
If every finite algebra in $\mathcal{V}_{\textrm{fin}}$
is $\langle 0,1\rangle$-minimal, then $\mathcal{V}$
is generated by the finite algebra $\m f_{\mathcal{V}}(2)$.
\end{clm}

\cproof
If every finite algebra in $\mathcal{V}_{\textrm{fin}}$
is $\langle 0,1\rangle$-minimal, then the free algebras
in $\mathcal V$ have the structure described in Lemmas~\ref{Gset_lm}
and \ref{Vect_lm}. Assume first that 
the finite members of $\mathcal{V}$
are $\langle 0,1\rangle$-minimal of type {\bf 1}.
Let $\mathcal{W}=\mathsf{H}\mathsf{S}\mathsf{P}(\m f_{\mathcal{V}}(2))$
be the subvariety of $\mathcal{V}$ that is generated by
$\m f_{\mathcal{V}}(2)$. Since $\mathcal{W}$ is a subvariety of
$\mathcal{V}$, every identity that holds in
$\mathcal{V}$ will hold in $\mathcal{W}$.
But since $\mathcal{V}$ is essentially unary, every
identity $s\approx t$ can depend on at most
two variables relative to $\mathcal{V}$,
so every identity that can be refuted in
$\mathcal{V}$ can be refuted in
$\m f_{\mathcal{V}}(2)$, which belongs to $\mathcal{W}$.
This is enough to show that
$\mathcal{V}$ and $\mathcal{W}$ satisfy
the same identities, hence are equal.
This concludes the proof that
$\mathcal{V}=\mathsf{H}\mathsf{S}\mathsf{P}(\m f_{\mathcal{V}}(2))$
holds when the finite members of $\mathcal{V}$
are $\langle 0,1\rangle$-minimal of type {\bf 1}.

Now, assume that 
the finite members of $\mathcal{V}$
are $\langle 0,1\rangle$-minimal of type {\bf 2}.
According to Lemma~\ref{Vect_lm}, there is a finite
field $\mathbb{F}$ and a minimal subvariety $\mathcal{U}$
of $\mathcal{V}$
such that  $\mathcal{U}$ is term equivalent to the variety of
all idempotent reducts of $\mathbb{F}$-spaces (Case~(1) of
Lemma~\ref{Vect_lm}), or 
$\mathcal{U}$ is term equivalent to the variety of
all $\mathbb{F}$-spaces (Case~(2) of
Lemma~\ref{Vect_lm}). In either Case, we proved
that, for any $\kappa$, we have
\[
\tag{$\ddagger$}
\m f_{\mathcal{V}}(\kappa)\cong \m f_{\mathcal{U}}(\kappa)
\times \m f_{\mathcal{V}}(i)
\]
where $i=1$ in Case~(1) and $i=0$ in Case~(2).
Since $\mathcal{U}$ is a minimal variety,
it is generated by any of its nontrivial members,
for example the algebra $\m f_{\mathcal{U}}(2)$.
Thus, we have
\begin{itemize}
\item $\m f_{\mathcal{V}}(2)$ generates $\m f_{\mathcal{U}}(2)$ and
  $\m f_{\mathcal{V}}(i)$ (read the isomorphism of ($\ddagger$)
  from left to right).
\item  $\m f_{\mathcal{U}}(2)$ and
  $\m f_{\mathcal{V}}(i)$ generate
  $\m f_{\mathcal{U}}(\kappa)$ and
  $\m f_{\mathcal{V}}(i)$ for any $\kappa$
  (ignore $\m f_{\mathcal{V}}(i)$ and use the
  fact that $\mathcal{U}$ is a minimal variety).
\item   $\m f_{\mathcal{U}}(\kappa)$ and
  $\m f_{\mathcal{V}}(i)$ generate  $\m f_{\mathcal{V}}(\kappa)$
  for any $\kappa$ (read the isomorphism of ($\ddagger$)
  from right to left).
\end{itemize}
Altogether, this shows that $\m f_{\mathcal{V}}(2)$
generates $\m f_{\mathcal{V}}(\kappa)$ for every $\kappa$,
by which we mean that
$\m f_{\mathcal{V}}(\kappa)\in\mathsf{H}\mathsf{S}\mathsf{P}(\m f_{\mathcal{V}}(2))$ for every $\kappa$. Hence 
$\mathcal{V}=\mathsf{H}\mathsf{S}\mathsf{P}(\m f_{\mathcal{V}}(2))$
when 
the finite members of $\mathcal{V}$
are $\langle 0,1\rangle$-minimal of type {\bf 2}.
\cqed
\bigskip

Claim~\ref{3implies4} proves that (3)$\Rightarrow$(4)
in this theorem.

Now assume that (4) holds, so that
$\mathcal V$ is generated by a finite, abelian
$\langle 0,1\rangle$-minimal
algebra $\m a$.
$\m a$ is polynomially equivalent to a $G$-set or to a
vector space. The property
of being polynomially equivalent
to a vector space or $G$-set is preserved
under the formation of finite powers,
nonempty subalgebras, and homomorphic images,
so all algebras in ${\mathcal V}_{\textrm{fin}}=
\mathsf{H}\mathsf{S}\mathsf{P}_{\textrm{fin}}(\m a)$
are $\langle 0,1\rangle$-minimal. This proves that (3) holds.
\end{proof}

\section{Closing Remarks} \label{closing}

The chief contribution of this paper
is the complete classification of Schreier
varieties in the locally finite setting.
Many earlier papers classified Schreier
subvarieties of given, well-known varieties.
Those earlier results did not provide enough
of a pattern to predict the results of this paper,
but now that we have a transparent
description of the locally finite Schreier varieties,
it is not hard to see some points of contact
with earlier results.

As we mentioned in the introduction,
Schreier proved in \cite{schreier}
that any subgroup of a free group is free.
The complete classification of Schreier varieties
of groups was completed by 
Neumann and Wiegold in \cite{neumann-wiegold}.
Their theorem is

\begin{thm} \label{NMx}
The Schreier varieties of groups are just 
\begin{enumerate}
\item[$(1)\hphantom{_p}$] the variety of all groups,
\item[$(2)\hphantom{_p}$] the variety of all abelian groups, and
\item[$(3)_p$] the varieties $\mathfrak{A}_p$ of abelian groups
  of exponent $p$ for each prime $p$.
\end{enumerate}
\end{thm}

Observe that the only varieties in Theorem~\ref{NMx}
that are locally finite are those in Item~$(3)_p$.
The variety $\mathfrak{A}_p$
in Item~$(3)_p$ consists of groups that are polynomially
equivalent to vector spaces over the $p$-element field.
Moreover, the varieties $\mathfrak{A}_p$
are exactly the nontrivial varieties
of groups that consist of algebras that are polynomially
equivalent to vector spaces.
In this way, one might consider
that this result foreshadows part of our theorem.
\bigskip

Evans classified the Schreier varieties of semigroups in
\cite{evans}. His theorem is:

\begin{thm} \label{TE}
The only Schreier varieties of semigroups are
\begin{enumerate}
\item the variety $Z_l$ of left-zero semigroups, defined by $xy = x$,
\item the variety $Z_r$ of right-zero semigroups, defined by $xy = y$,
\item the variety $C$ of constant semigroups, defined by $xy = zt$.
\item the varieties $A_p$ of abelian groups
  satisfying $x^p = 1$, $p$ prime.
\end{enumerate}  
\end{thm}

Recall that the variety of semigroups has
no zeroary symbols, so by our convention concerning
constants versus pseudoconstants a Schreier variety
of semigroups should have no pseudoconstants.
Evans chooses instead to
call a variety of semigroups ``Schreier'' if
whenever $|\m s|>1$ and $\m s$ is a
subsemigroup of a free semigroup, then $\m s$
is free. This allows a Schreier variety of semigroups
to have at most one pseudoconstant.
Evans explains his convention as follows:
\smallskip

\begin{quote}
{\bf Remark.} {\it There is a particular reason for inserting the
    word \emph{non-trivial} in
our definition of Schreier variety. Without it we would be forced to classify
the variety of constant semigroups $C$ as non-Schreier because any free
semigroup $F_k(C)$ in this variety contains an idempotent element and this
one-generator subsemigroup is not $F_1(C)$.
However, all non-trivial subsemigroups
of $F_k(C)$ are free in $C$. What is worse,
we would also have to classify the varieties
$A_p$ as non-Schreier since again as a semigroup,
$F_k(A_p)$ contains a one-element
subsemigroup. This difficulty is avoided in groups,
of course, by regarding one-element groups as
free groups on an empty set of generators.}
\end{quote}
\smallskip

Somewhat surprisingly, all of the varieties
identified by Evans are locally finite.
Moreover, his examples which have a pseudoconstant
and no constant can be adapted by expanding by
a constant symbol so that they meet our definition
of Schreier.
The first three types
of varieties from Theorem~\ref{TE}
then become examples of varieties whose finite
members are polynomially equivalent to $G$-sets
for the group $G=\{1\}$.
The fourth example is the list of the varieties of semigroups
whose finite members are polynomially equivalent
to vector spaces over the $p$-element field
for appropriate $p$.
\bigskip

It is proved by Shirshov in \cite{shirshov} that, if $\m f$ is a field,
then the variety of all Lie algebras over $\m f$
is a Schreier variety. The full classification
of Schreier subvarieties of Lie algebras over $\m f$ was given by 
Za\u{\i}cev in \cite{zaitsev}. His theorem is:

\begin{thm} \label{NM}
  Let $\m f$ be a field. The nontrivial
  Schreier varieties of Lie algebras over $\m f$ are 
\begin{enumerate}
\item the variety of all Lie algebras over $\m f$, and
\item the variety of all abelian Lie algebras over $\m f$.
\end{enumerate}
\end{thm}

Observe that the variety of abelian Lie algebras over $\m f$
satisfies the identity $[x,y]\approx 0$, and $0$ is a constant
of the language, so we can delete the
Lie bracket from the signature
without changing the variety up to
term equivalence. Hence, the variety of abelian Lie
algebras over $\m f$
consists of algebras that are polynomially
equivalent to $\m f$-vector spaces.
This variety will be locally finite if
$\m f$ is a finite field.
Thus, we might also accept this theorem
as providing a hint about the classification of Schreier
varieties in the locally finite setting.
\bigskip

In \cite{skornyakov}, Skornyakov classified Schreier varieties of
$M$-sets ( = sets equipped with an action of a fixed monoid $M$).
A variety of $M$-sets has no constants,
but Skornyakov allows pseudoconstants.
If we rephrase his theorem in a way that accords
with our convention that a Schreier variety with
no constants should have no pseudoconstants, then we get

\begin{thm} \label{NM}
Let $M$ be a nontrivial monoid. The 
variety $\mathcal{V}_M$ of $M$-sets is Schreier
if and only if the following hold:
\begin{enumerate}
\item any principal left ideal of $M$ has a right-cancellative generator,
\item any two incomparable left ideals of $M$ are disjoint, and
\item $M$ satisfies the ACC on principal left ideals.
\end{enumerate}
\end{thm}

Observe that if $M$ satisfies these conditions
and $a\in M$, then the left ideal
$Ma$ must have a right cancellative generator, say $a'$.
The statement that $a'$ generates
$Ma$ means that $Ma=Ma'$. The statement that
$a'$ is right-cancellative means
that the right multiplication
map $\rho_{a'}\colon M\to M\colon x\mapsto xa'$
is injective. Now, $\mathcal{V}_M$ will be locally finite
if and only if $M$ is finite. In this case,
$a'$ will be right-cancellative if and only if
$\rho_{a'}$ is a permutation of the finite set $M$,
which will hold 
if and only if the element
$a'$ has a finite power equal to $1\in M$.
This forces $a'$ to be a unit of $M$. Altogether,
when $M$ is finite, the right-cancellative
generator $a'\in Ma$ must be a unit, in which case
$Ma=Ma' = M$.
The conclusion is that
$M$ itself is the only principal left ideal of $M$.
This makes every element of $M$ left-invertible, hence
invertible, hence $M$ must be a (finite) group. When $M$ is a group
both Items (2) and (3) hold.
Thus, 
Skornyakov's Theorem proves that a locally finite
variety of $M$-sets is Schreier (in our sense) if and only
if it is a variety of $G$ sets for the finite group $G=M$.
Again, this earlier theorem might be considered
to foreshadow our main result.

\bibliographystyle{plain}

\end{document}